\documentclass[12pt]{article}

\usepackage{amsfonts}
\usepackage{amssymb}
\usepackage{amsthm}
\usepackage{amsmath}
\usepackage{amscd}
\usepackage{mathrsfs}
\usepackage{enumerate}
\usepackage{ulem}
\usepackage{fontenc}
\usepackage[all]{xy}
\usepackage{array}
\usepackage[english]{babel}

\def\CC {{\mathbb C}}     
\def\NN {{\mathbb N}}     
\def\PP {{\mathbb P}}     
\def\QQ {{\mathbb Q}}     
\def\RR {{\mathbb R}}     
\def\ZZ {{\mathbb Z}}     

\def\tst {\Longleftrightarrow}
\def\ul  {\underline}

\newtheorem{theorem}{Theorem}[section]
\newtheorem{lemma}[theorem]{Lemma}

\newtheorem{coro}[theorem]{Corollary}
\newtheorem{rem}{Remark}[section]

\newtheorem{example}{Example}[section]

\begin{document}

\title{Partial convex hulls of coadjoint orbits\\ and degrees of invariants}

\author{Valdemar V. Tsanov \footnote{Work supported by DFG grant SFB-TR191, ``Symplectic Structures in Geometry, Algebra and Dynamics'' and the Bulgarian Ministry of Education and Science, Scientific Programme ``Enhancing the Research Capacity in Mathematical Sciences (PIKOM)'', No. DO1-67/05.05.2022.}}

\maketitle

\abstract{We study properties of convex hulls of (co)adjoint orbits of compact groups, with applications to invariant theory and tensor product decompositions. The notion of partial convex hulls is introduced and applied to define two numerical invariants of a coadjoint orbit of a semisimple connected compact Lie group. It is shown that the orbits, where any one of these invariants does not exceed a given number $r$, form, upon intersection with a fixed Weyl chamber, a rational convex polyhedral cone in that chamber, related to the Littlewood-Richardson cone of the $r$-fold diagonal embedding of $K$. The numerical invariants are shown to provide lower bounds for degrees of invariant polynomials on irreducible unitary representations.}

\section{Introduction and main results}

Let $K$ be a semisimple connected compact Lie group with Lie algebra $\mathfrak k$. The negative of The Killing form of $\mathfrak k$ is positive definite, denoted by $(|)$; it provides an isomorphism $\mathfrak k\cong \mathfrak k^*$. The (co)adjoint $K$-orbits in $\mathfrak k$ are parametrized by any fixed Weyl chamber $\mathfrak t_+$ in a maximal abelian subalgebra $\mathfrak t\subset \mathfrak k$. For $\lambda\in\mathfrak t_+$, the convex hull ${\rm Conv}(K\lambda)$ is also called the orbitope of $\lambda$; it was studied by Biliotti, Ghigi and Heinzner in \cite{Biliotti-Ghigi-Heinzner}, where all faces of this orbitope are described and shown to be exposed. In this article, we introduce the notion of partial convex hulls and derive two numerical invariants ${\bf r}_0(\lambda)$ and ${\bf r}(\lambda)$. We show that the locus where any one of these invariants does not exceed a given $r\in\NN$ forms a rational polyhedral convex cone in $\mathfrak t_+$ related to the Littlewood-Richardson cone of the $r$-fold diagonal embedding $K\hookrightarrow K^{\times r}$. The consideration of ${\bf r}_0(\lambda)$ and ${\bf r}(\lambda)$ is motivated by interpretations in terms of momentum maps and resulting applications to invariant theory and decompositions of tensor products of irreducible representations. The main results are formulated after the introduction of some notions and notation.\\

The $r$-th partial convex hull of a subset $X\subset E$ of a Euclidean space $E$, for a positive integer $r$, is defined as
\begin{gather}\label{For CrX}
C_r(X) := \bigcup\limits_{x_1,...,x_r\in X} {\rm Conv}\{x_1,...,x_r\} \;.
\end{gather}
We focus on $X=K\lambda\subset \mathfrak k$, a (co)adjoint $K$-orbit through an arbitrary $\lambda\in \mathfrak t_+$. In such a case, $C_r(K\lambda)$ is preserved by $K$ and is equal to the convex hull ${\rm Conv}(K\lambda)$ for sufficiently large $r$, since (see Lemma \ref{Lemma ConvX is Cellplus1})
$$
{\rm Conv}(K\lambda) = C_{\ell+1}(K\lambda) \;, \quad {\rm where}\quad \ell:=\dim \mathfrak t \;.
$$
The minimal $r$ for which $C_r(K\lambda)$ is convex will be denoted by
$$
{\bf r}(\lambda) := \min\{r\in\NN: {\rm Conv}(K\lambda) = C_r(K\lambda)\} \;.
$$
Since $K$ is semisimple, $0\in{\rm Conv}(K\lambda)$ for all $\lambda$, and we define
$$
{\bf r}_0(\lambda) := \min\{r\in\NN: 0\in C_r(K\lambda) \} \;.
$$


Let $T={\rm exp}(\mathfrak t)\subset K$ be the maximal torus in $K$ corresponding to $\mathfrak t$. The character lattice ${\rm Hom}(T,\CC^\times)$ is naturally embedded in $i\mathfrak t^*$; we denote by $\Lambda\subset \mathfrak t$ its image under $\lambda\mapsto-i\lambda$ composed with the identification $\mathfrak t^*\cong \mathfrak t$ provided by $(|)$, and refer to $\Lambda$ as the weight lattice of $T$. Let $\Lambda^+=\Lambda\cap \mathfrak t_+$ be the monoid of dominant weights. Then $\Lambda^+$ parametrizes the irreducible complex representations of $K$, and we denote by $V_\lambda$ the irreducible representation with highest weight $i(\lambda|\cdot)$ with $\lambda\in\Lambda^+$.

Let $\CC[V_\lambda]$ be the polynomial ring on $V_\lambda$ and $\CC[V_\lambda]_d$ denote the space of homogeneous polynomals of degree $d$. The ring of $K$-invariant polynomials $\CC[V_\lambda]^K$ is finitely generated, by Hilbert's theorem, and the ordered sequence of the degrees $0=d_0(\lambda)\leq ...\leq d_p(\lambda)$ of a minimal set of homogeneous generators is the same for all such sets and thus canonically determined by $\lambda\in\Lambda^+$. Hilbert's theorem is famously nonconstructive and the degrees $d_{j}(\lambda)$ are generally unknown, although many special cases have subjected to extensive studies and classifications. Upper bounds for the Noether number $d_p(\lambda)$ have been derived by Popov \cite{Popov-Bounds} and Derksen \cite{Derksen-PolyBound}. Here we study lower bounds for the minimal positive degree $d_1(\lambda)=\min\{d\in\ZZ_{>0}:\CC[V_\lambda]_d^K\ne 0\}$, whenever it exists, and the variations of this degree and our bounds along variations of $\lambda$ in $\Lambda^+$. The bounds are related to the convex geometry of the orbit $K\lambda\subset \mathfrak k$.

The Littlewood-Richardson monoid $\mathcal{LR}_r$ and cone $\mathcal{CLR}_r$ of the $r$-fold diagonal embedding $\varphi_r:K\hookrightarrow K^{\times r}$ are defined as
\begin{gather}\label{For LR i CLR}
\begin{array}{rl}
\mathcal{LR}_r &=\mathcal{LR}_r(K) = \{(\lambda_1,...,\lambda_r)\in(\Lambda^+)^{\times r}: (V_{\lambda_1}\otimes...\otimes V_{\lambda_r})^K\ne 0 \} \\
\mathcal{CLR}_r &=\mathcal{CLR}_r(K) = {\rm Conv}(\mathcal{LR}_r)\\
& = \{(\lambda_1,...,\lambda_r)\in(\mathfrak t_+)^{\times r}:0\in K\lambda_1+...+K\lambda_r\}\;.
\end{array}
\end{gather}
It is well known, see \cite{Brion-Survey} for a survey of the key results and historical references, that $\mathcal{LR}_r$ is a finitely generated monoid, $\mathcal{CLR}_r$ is a rational polyhedral convex cone, and $\mathcal{LR}_r$ is of finite index in the finitely generated monoid $\Lambda^{\times r}\cap\mathcal{CLR}_r$ of integral points in the cone.

The following three theorems are the main results of this article.

\begin{theorem}\label{Theo Theo r0 below d1}
The inequality ${\bf r}_0(\lambda)\leq d_1(q\lambda)$ holds for all $\lambda\in\Lambda^+\setminus\{0\}$ and $q\in \QQ_{>0}$ such that $q\lambda\in\Lambda^+$.
\end{theorem}

Let $\iota_r:\mathfrak k\to\mathfrak k^{\oplus r}, x\mapsto(x,...,x)$ be the differential of $\varphi_r$.
 
\begin{theorem}\label{Theo Theo Cr cones}
For any coadjoint orbit $K\lambda$ and $r\in\NN$, $0\in C_r(K\lambda)$ if and only if $0\in K\lambda+...+K\lambda$ ($r$-fold sum). The set $\mathfrak A_r:=\{\lambda\in\mathfrak t_+: {\bf r}_0(\lambda)\leq r\}$ is a rational polyhedral convex cone in $\mathfrak t_+$ satisfying 
$\iota_r(\mathfrak A_r)= \mathcal{CLR}_r \cap \iota_r(\mathfrak t_+)$. 
\end{theorem}

\begin{theorem}\label{Theo Theo Er cones}
The set $\mathfrak C_r:=\{\lambda\in\mathfrak t_+: {\bf r}(\lambda)\leq r\}$ is a rational polyhedral convex cone in $\mathfrak t_+$. For $\lambda\in\mathfrak t_+$ the following are equivalent:
\begin{enumerate}
\item[\rm (i)] ${\bf r}(\lambda)\leq r$, i.e., $C_r(K\lambda)$ is convex;
\item[\rm (ii)] the $r$-fold sum $K\lambda+...+K\lambda$ is convex;
\item[\rm (iii)] for every $\xi\in\mathfrak t_+$, $\lambda_{\vert\mathfrak k'_\xi}\in\mathcal{CLR}_{r}(K'_\xi)$, where $K'_\xi$ is the derived subgroup of the centralizer of $\xi$, $\mathfrak k'_\xi$ is its Lie algebra and $\lambda_{\vert\mathfrak k'_\xi}$ is the orthogonal projection of $\lambda$ to $\mathfrak k'_\xi$.
\end{enumerate}
\end{theorem}

Further details and corollaries from the above theorems will be given in the main text. In \S\ref{Sect Moment i CLR} we recall the interpretation of the Littlewood-Richardson cone in terms of momentum maps, we derive properties of ${\bf r}_0$ and prove Theorems \ref{Theo Theo r0 below d1} and \ref{Theo Theo Cr cones}. In \S\ref{Sect PartConv} we describe a relationship between ${\bf r}_0$ and ${\bf r}$, prove Theorem \ref{Theo Theo Er cones} and compute the values of ${\bf r}_0$ and ${\bf r}$ at fundamental weights of classical groups.\\

\noindent{\bf Acknowledgement:} A substantial part of the work on this article was done at the Ruhr-Universit\"at Bochum, with the support of DFG grant SFB-TR191, ``Symplectic Structures in Geometry, Algebra and Dynamics''. The author is grateful to St{\'e}phanie Cupit-Foutou and Peter Heinzner for helpful discussions and support.

\section{The momentum map and the Littlewood-Richardson cone}\label{Sect Moment i CLR}

We begin by recalling Heckman's interpretation of projections of coadjoint orbits as momentum maps, \cite{Heckman}, and the resulting interpretation of the Littlewood-Richardson monoid and cone in the framework of the Geometric Invariant Theory (GIT) of Hilbert-Mumford and Kirwan, \cite{Kirwan}. The survey \cite{Brion-Survey} contains more details and historical references for the background results outlined here. The first basic fact is that every coadjoint orbit $K\lambda\subset \mathfrak k^*\cong \mathfrak k$ of a compact connected Lie group $K$ admits a canonical invariant K\"ahler structure, called the Kostant-Kirillov-Sourieau structure, so that the inclusion $K\lambda\subset \mathfrak k^*$ is a $K$-equivariant momentum map. The underlying complex manifolds of the coadjoint orbits are known as the flag varieties of $K$. Two orbits are equivariantly isomorphic as complex manifolds if and only if their points of intersection with a fixed Weyl chamber $\mathfrak t_+$ belong to the relative interior of the same face of $\mathfrak t_+$. By the Borel-Weil theorem, for integral $\lambda\in\Lambda^+$, the K\"ahler structure corresponds to a $K$-linearized holomorphic line bundle $\mathcal L_\lambda$ with space of global sections isomorphic to the irreducible representation $V_{\lambda}^*$. Moreover $K\lambda$ is equivariantly isomorphic to the projective orbit of a highest weight vector $K[v_\lambda]\subset\PP(V_\lambda)$ with the K\"ahler structure induced by the Fubini-Study form of a suitably normalized $K$-invariant Hermitean form on $V_\lambda$. The homogeneous coordinate ring of the projective variety $K[v_\lambda]$, i.e., the ring of sections of $\mathcal L_\lambda$, is
$$
R_\lambda:=\bigoplus\limits_{q\in \ZZ_{\geq 0}} H^0(K\lambda,\mathcal L_{\lambda}^q) \cong \bigoplus\limits_{q\in \ZZ_{\geq 0}} V_{q\lambda}^* \;.
$$
Heckman's theorem, \cite{Heckman}, states that, if $K \subset \tilde K$ is an embedding of compact connected Lie groups, $\iota:\mathfrak k\hookrightarrow\tilde{\mathfrak k}$ is the corresponding inclusion of Lie algebras, and $\iota^*:\tilde{\mathfrak k}^*\to \mathfrak k^*$ is the dual map, then the restriction of $\iota^*$ to the coadjoint $\tilde K$-orbit $Z_{\tilde\lambda}=\tilde K\tilde\lambda$ through any fixed $\tilde\lambda\in\tilde{\mathfrak k}^*$,
$$
\mu=\mu_K^{\tilde\lambda}:=(\iota_r^*)_{\vert_{Z_{\tilde\lambda}}}: Z_{\tilde\lambda}\to \mathfrak k^*\cong \mathfrak k
$$
is a $K$-equivariant momentum map with respect to the canonical $\tilde K$-equivariant Kostant-Kirillov-Sourieau structure on $Z_{\tilde\lambda}$. This puts us in the setting of Kirwan's Geometric Invariant Theory, \cite{Kirwan}, and we have the following two fundamental results. First, the intersection of the momentum image with a Weyl chamber, $\mu(Z_{\tilde\lambda})\cap\mathfrak t_+$, is a convex polytope, called the momentum polytope. Second, for integral $\tilde\lambda\in\tilde\Lambda^+$, the symplectic reduction $\mu^{-1}(0)/K$ can be identified with the GIT-quotient, which is an algebraic variety isomorphic to the projective spectrum of the invariant ring $(R_{\tilde\lambda})^K$. This invariant ring can in turn be interpreted by applying the Borel-Weil theorem to $\tilde K$:
$$
\mu^{-1}(0)/K \cong {\rm Proj}(R_{\tilde\lambda}^K) \;,\quad R_{\tilde\lambda}^K\cong  \bigoplus\limits_{q\in \ZZ_{\geq 0}} (V_{q\tilde\lambda}^*)^K \;.
$$
Since $V_{\tilde\lambda}^K\ne 0$ if and only if $(V_{\tilde\lambda}^*)^K\ne 0$, we have the equivalences
\begin{gather}\label{For 0inmu VKne0}
0\in \mu(Z_{\tilde\lambda}) \;\tst\; R_{\tilde\lambda}^K \ne \CC \; \tst\; \exists q\in\NN: (V_{q\tilde\lambda})^K\ne 0 \;.
\end{gather}
The set
$$
\mathcal{CLR}(K\subset \tilde K) = \{\tilde\lambda\in\tilde{\mathfrak t}_+:0\in\mu(Z_{\tilde\lambda})\}
$$
is a rational polyhedral convex cone, known as the genaralized Littlewood-Richardson cone, described by inequalities derived from the Hilbert-Mumford functions of suitable one-parameter subgroups of $K$. These inequalities have a long history culminating with Ressayre's characterization of a minimal list of inequalities, \cite{Ressayre-GITandEigen}. We shall not need the exact form of the inequalities, and satisfy ourselves with the aforementioned properties of the cone. The set of integral points $\tilde\Lambda\cap\mathcal{CLR}(K\subset\tilde K)$ is a finitely generated monoid containing as a finitely generated submonoid of finite index the so-called generalized Littlewood-Richardson monoid
$$
\mathcal{LR}(K\subset\tilde K) = \{\tilde\lambda\in\tilde\Lambda^+: (V_{\tilde\lambda})^K\ne 0\} \;.
$$
Thus there exists $q\in\NN$ such that $q(\tilde\Lambda\cap\mathcal{CLR}(K\subset\tilde K))\subset\mathcal{LR}(K\subset\tilde K)$. Furthermore, we have $\mathcal{CLR}(K\subset\tilde K)={\rm Conv}(\mathcal{LR}(K\subset\tilde K))$. We refer the reader to \cite{Brion-Survey} for more details.

Here we consider what is in fact a prototypical case for the above construction: the diagonal subgroup $K\subset\tilde K:=K^{\times r}$ of a $r$-fold Cartesian product for some fixed $r\in\NN$. On the level of Lie algebras we have the inclusion $\iota_r:\mathfrak k\hookrightarrow\tilde{\mathfrak k}:=\mathfrak k^{\oplus r}$, $x\mapsto(x,...,x)$. Every Weyl chamber $\mathfrak t_+$ of $\mathfrak k$ is contained in unique Weyl chamber of $\tilde{\mathfrak k}$, namely $\tilde{\mathfrak t}_+=(\mathfrak t_+)^{\times r}$. We also have $\tilde\Lambda=\Lambda^{\times r}$. The irreducible representations of $\tilde K$ are tensor products of irreducible representations of $K$, i.e., for $\tilde\lambda=(\lambda_1,...,\lambda_r)\in\tilde\Lambda^+$, $V_{\tilde\lambda}=V_{\lambda_1}\otimes...\otimes V_{\lambda_r}$. Therefore, the space $V_{\tilde\lambda}^K$ of $K$-invariant vectors in an irreducible $\tilde K$-module is, in this setting, the space of invariant tensors in a tensor product of $r$ irreducible $K$-modules.

For $\tilde\lambda=(\lambda_1,...,\lambda_r)\in\tilde{\mathfrak t}_+$ denote by $Z_{\tilde\lambda}=\tilde K\tilde\lambda=K\lambda_1\times...\times K\lambda_r \subset \tilde{\mathfrak k}$ the (co)adjoint $\tilde K$-orbit through $\tilde\lambda$ equipped with its canonical Kostant-Kirillov-Sourieau K\"ahler structure. 

To apply Heckman's theorem, we note that the dual map $\iota_r^*$ is given by
$$
\iota_r^*:\tilde{\mathfrak k}\cong\tilde{\mathfrak k}^*\cong(\mathfrak k^*)^{\oplus r}\to \mathfrak k^*\cong \mathfrak k, \quad (x_1,...,x_r)\mapsto x_1+...+x_r \;.
$$
and so
$$
\mu=\mu_K^{\tilde\lambda}=(\iota_r^*)_{\vert_{Z_{\tilde\lambda}}}: Z_{\tilde\lambda}\to \mathfrak k
$$
is a momentum map for the diagonal $K$-action on $Z_{\tilde\lambda}$. The momentum image is 
$$
\mu(Z_{\tilde\lambda})=K\lambda_1+K\lambda_2+...+K\lambda_r \;.
$$
Thus we can apply the GIT framework explained above. First, by Kirwan's theorem, $\mu(Z_{\tilde\lambda})\cap\mathfrak t_+$ is a convex polytope. Second, we can interpret the Littlewood-Richardson cone $\mathcal{CLR}_r$ defined in the Introduction as
$$
\mathcal{CLR}_r=\mathcal{CLR}(K\subset K^{\times r})=\{(\lambda_1,...,\lambda_r)\in(\mathfrak t_+)^{\times r}: 0\in\mu(Z_{(\lambda_1,...,\lambda_r)})\} \;.
$$
The set of integral points of this cone is
$$
\tilde\Lambda\cap\mathcal{CLR}_r = \{(\lambda_1,...,\lambda_r)\in(\Lambda^+)^{\times r}:\exists q\in\NN, (V_{q\lambda_1}\otimes...\otimes V_{q\lambda_r})^K\ne 0\} \;.
$$
and is a finitely generated monoid containing $\mathcal{LR}_r$ as a submonoid of finite index. In particular, for $\tilde\lambda=(\lambda_1,...,\lambda_r)\in\tilde\Lambda^+$, the equivalence (\ref{For 0inmu VKne0}) takes the form
$$
0\in \mu(Z_{\tilde\lambda}) \;\tst\; \exists q\in\NN:\;(V_{q\lambda_1}\otimes V_{q\lambda_2}\otimes ...\otimes V_{q\lambda_r})^K\ne 0 \;.
$$
With this preparation, we turn to our applications.

\begin{theorem}\label{Theo Cr cones}
For any (co)adjoint orbit $K\lambda\in\mathfrak k$ and $r\in\NN$, $0\in C_r(K\lambda)$ if and only if $0\in K\lambda+...+K\lambda\subset\mathfrak k$ ($r$-fold sum). Furthermore, the set $\mathfrak A_r:=\{\lambda\in\mathfrak t_+: {\bf r}_0(\lambda)\leq r\}$ is a rational polyhedral convex cone in $\mathfrak t_+$ satisfying 
$\iota_r(\mathfrak A_r)= \mathcal{CLR}_r \cap \iota_r(\mathfrak t_+)$. 
\end{theorem}

\begin{proof}
We consider the $\tilde K$-homogeneous complex manifold $Z=(K/K_\lambda)^{\times r}$, which is $\tilde K$-equivariantly isomorphic to $(K\lambda)^{\times r}$ as a complex manifold. This manifold can be endowed with different K\"ahler structures and $K$-momentum maps arising from the different $\tilde K$-equivariant embeddings of $Z$ into $\tilde{\mathfrak k}$; these embedding are obtained as $Z\cong \tilde K\tilde\lambda'\subset\tilde{\mathfrak k}$ for $\tilde\lambda'$ belonging to the relative interior of the same face of $\tilde{\mathfrak t}_+$ as $(\lambda,...,\lambda)$. It is convenient to take into consideration cases, where instead of an embedding of $Z$ we have a map whose image is isomorphic to $(K/K_{\lambda})^{\times s}$ for some $s\leq r$. For $\lambda\in\mathfrak t_+$ and $\ul{q}=(q_1,...,q_r)\in(\RR_{\geq 0})^{\times r}\setminus\{0\}$, we denote $\ul{q}\cdot\lambda := (q_1\lambda,...,q_r\lambda)\in \tilde{\mathfrak t}_+$ and $Z_{\ul{q}\cdot\lambda}:=\tilde K\tilde\lambda=q_1K\lambda\times...\times q_rK\lambda \subset \tilde{\mathfrak k}$. Then, by Heckman's theorem,
$$
\mu^{\ul{q}\cdot\lambda}:=\frac{1}{\sum q_j}(\iota_r^*)_{\vert_{Z_{\ul{q}\cdot\lambda}}}: Z_{\ul{q}\cdot\lambda}\to  \mathfrak k^* \cong\mathfrak k
$$
is a momentum map for the diagonal $K$-action on $Z_{\ul{q}\cdot\lambda}$ and the Kostant-Kirillov-Sourieau K\"ahler structure multiplied by a factor of $\frac{1}{\sum q_j}$. We introduce this factor, because we want the image of $\mu^{\ul{q}\cdot\lambda}$ to be contained in the partial convex hull $C_r(K\lambda)$ defined in (\ref{For CrX}). We observe that $C_r(K\lambda)$ consists of the images of all such momentum maps, i.e., 
$$
C_r(K\lambda)=\bigcup\limits_{\ul{q}\in(\RR_{\geq 0})^{\times r}\setminus\{0\}}\mu^{\ul{q}\cdot\lambda}(Z_{\ul{q}\cdot\lambda}) \;.
$$
Hence, the condition $0\in C_r(K\lambda)$ can now be reformulated in the following equivalent forms:

(a) $r\geq {\bf r}_0(\lambda)$;

(b) $\exists\ul{q}=(q_1,...,q_r)\in(\RR_{\geq 0})^{\times r}\setminus\{0\}$ such that $0\in \mu^{\ul{q}\cdot\lambda}(Z_{\ul{q}\cdot\lambda})$;

(c) $\exists\ul{q}=(q_1,...,q_r)\in(\RR_{\geq 0})^{\times r}\setminus\{0\}$ such that $(q_1\lambda,...,q_r\lambda)\in\mathcal{CLR}_r$.

Let us notice that the convex cone $\mathcal{CLR}_r\subset (\mathfrak t_+)^{\times r}$ is stable under permutations of the components, as can be seen directly from the definition (\ref{For LR i CLR}). This implies that condition (c) is equivalent to $(\lambda,...,\lambda)\in\mathcal{CLR}_r$. Indeed, if such $\ul{q}$ exists, then all permutations $(q_{\sigma(1)\lambda},...,q_{\sigma(r)}\lambda)$ belong to $\mathcal{CLR}_r$ and so does their sum, but this sum is positively proportional to $(\lambda,...,\lambda)$, so $(\lambda,...,\lambda)\in\mathcal{CLR}_r$. We obtain
$$
r\geq {\bf r}_0(\lambda) \;\tst\; (\lambda,...,\lambda)\in\mathcal{CLR}_r \;,
$$
which completes the proof of the theorem.
\end{proof}

Next, we derive a corollary concerning the degrees of $K$-invariant elements of the polynomial ring $\CC[V_{\lambda}]$ over the irreducible representation space $V_\lambda$ of $K$. Recall that $\CC[V_{\lambda}]_d$ denotes the space of homogeneous polynomials of degree $d$; this space can be interpreted as the $d$-th symmetric tensor power of the dual $K$-module $V_\lambda^*$. Furthermore, we have $(S^dV_\lambda)^*\cong S^d(V_\lambda^*)$. The dual module $V_\lambda^*$ is irreducible and we denote by $\lambda^*$ its highest weight, so that $V_\lambda^*\cong V_{\lambda^*}$.

\begin{coro}\label{Coro Theo b1 leq r0}
Let $\lambda\in\Lambda^+\setminus\{0\}$ and $b_1(\lambda)=\min\{b\in\NN:\exists q\in\NN:\CC[V_{q\lambda}]_b^K\ne 0\}$. Then the inequality $b_1(\lambda)\geq r_0(\lambda)$ holds.
\end{coro}

\begin{proof} According to the remarks leading to the corollary, we have $\CC[V_{\lambda}]_r=S^rV_{\lambda}^*$ and $\CC[V_{\lambda}]_r^*=S^rV_{\lambda}\cong \CC[V_{\lambda^*}]_r$. Since any given $K$-module has $K$-invariant vectors if and only if its dual has $K$-invariant vectors, we have $b_1(\lambda)=b_1(\lambda^*)$ and we may consider $S^rV_{\lambda}$ instead of $\CC[V_{\lambda}]_r$. Now, the existence of $q\in\NN$ such that $(S^rV_{q\lambda})^K\ne 0$ clearly implies the existence of $q\in\NN$ such that $(V_{q\lambda}^{\otimes r})^K\ne 0$, which in turn is equivalent to $(\lambda,...,\lambda)\in\mathcal{CLR}_r$. But we have shown in Theorem \ref{Theo Cr cones} that the latter is equivalent to $r\geq {\bf r}_0(\lambda)$, and this completes the proof.
\end{proof}

\begin{rem}
Let $\lambda\in\mathfrak t_+$, $r\in\NN$ and 
$$\mathfrak Q_{\lambda,r}:=\{(q_1\lambda,...,q_r\lambda)\in(\mathfrak t_+)^{\times r}:q_j\in\RR_{\geq0}\forall j\},$$
which is a simplicial cone in $(\mathfrak t_+)^{\times r}$. Then $\mathfrak Q_{\lambda,r}\nsubseteq \mathcal{CLR}_r$ and
$$r\geq {\bf r}_0(\lambda) \;\tst\; \mathfrak Q_{\lambda,r}\cap \mathcal{CLR}_r\ne \{0\} \;\tst\; \mathfrak Q_{\lambda,r}\cap \partial(\mathcal{CLR}_r)\ne \{0\}.$$ 
Consequently, if $\lambda\in\Lambda^{++}:=\Lambda\cap {\rm Relint}(\mathfrak t_+)$ is strictly dominant and $r={\bf r}_0(\lambda)$, then $\mathfrak Q_{\lambda,r}$ intersects a regular face of $\mathcal{CLR}_{r}$. It would be interesting to understand which regular faces of $\mathcal{CLR}_{r}$ can be attained in this way, i.e., intersect $\mathfrak Q_{\lambda,r}$ for some $\lambda\in\Lambda^{++}$.
\end{rem}

\section{Coadjoint orbitopes and partial convex hulls}\label{Sect PartConv}

Here we establish some basic properties of the partial convex hulls $C_r(K\lambda)$ of coadjoint $K$-orbits, derive combinatorial interpretations for ${\bf r}_0(\lambda)$ and ${\bf r}(\lambda)$, and compute the values for all fundamental weights of classical groups. We also recall the description of the faces of the orbitope ${\rm Conv}(K\lambda)$ by Biliotti, Ghigi and Heinzner, \cite{Biliotti-Ghigi-Heinzner}, as it is used in an essential way in the our approach to ${\bf r}(\lambda)$.

Let $\Delta\subset\Lambda$ be the root system of $K$ with respect to $T$, split as $\Delta=\Delta^+\sqcup\Delta^-$ by the chosen Weyl chamber $\mathfrak t_+$, and let $\Pi\subset\Delta^+$ be the set of simple roots. Let $W=N_K(T)/T$ be the Weyl group acting on $\mathfrak t$ as the group generated by the (simple) root reflections. Recall that $K\lambda\cap \mathfrak t = W\lambda$ for every $\lambda\in\mathfrak t$. By a classical theorem of Kostant, if $\iota_T^*:\mathfrak k\to\mathfrak t$ is the orthogonal projection, then $\iota_T^*(K\lambda)={\rm Conv}(W\lambda)={\rm Conv}(K\lambda)\cap\mathfrak t$. Consequently, 
$$
{\rm Conv}(K\lambda)=K{\rm Conv}(W\lambda) \;.
$$
It turn out that all faces of ${\rm Conv}(K\lambda)$ arise from faces of the polytope ${\rm Conv}(W\lambda)$. We continue with necessary notation.

The complexified Lie algebra $\mathfrak k^c=\mathfrak k\oplus i\mathfrak k$ is a complex semisimple Lie algebra with root space decomposition: 
$$
\mathfrak k^c = \mathfrak t^c\oplus(\bigoplus\limits_{\alpha\in\Delta} \CC e_\alpha)\;.
$$
The root vectors can be chosen so that 
$$
\mathfrak k = \mathfrak t\oplus(\bigoplus\limits_{\alpha\in\Delta^+} (\RR(e_\alpha-e_{-\alpha})\oplus \RR i(e_\alpha+e_{-\alpha}))) \;.
$$
For $\xi\in\mathfrak t$, the centralizer subalgebra $\mathfrak k_\xi=\mathfrak z_{\mathfrak k}(\xi)$ is given by
$$
\mathfrak k_\xi = \mathfrak t \oplus \mathfrak k\cap(\bigoplus\limits_{\alpha\in\Delta:(\alpha|\xi)=0} \CC e_\alpha) \;,\quad  \;.
$$
The corresponding centralizer subgroup of $K$ is connected, and we have $K_\xi = Z_K(\xi) = {\rm exp}(\mathfrak k_\xi)$. For $\xi\in\mathfrak t_+$, $\mathfrak k_\xi$ is uniquely determined by the set of simple roots vanishing on $\xi$, $\Pi^\xi=\{\alpha\in\Pi:(\alpha|\xi)=0\}$. Conversely, any subset $\hat\Pi\subset\Pi$ gives rise to a subalgebra $\mathfrak k_{\hat\Pi}\subset\mathfrak k$ with center $\mathfrak z_{(\hat\Pi)}:=\mathfrak z(\mathfrak k_{(\hat\Pi)})=\cap_{\alpha\in\hat\Pi}{\rm ker}\alpha$.

\begin{theorem}\label{Theo BGH ConvKlambda} (Biliotti-Ghigi-Heinzner \cite{Biliotti-Ghigi-Heinzner}) Let $\lambda\in \mathfrak t_+$.

Then ${\rm Conv}(K\lambda)=K{\rm Conv}(W\lambda)$ holds, the faces of ${\rm Conv}(K\lambda)$ are exposed and are exactly the subsets of the form
$$
{\rm Conv}(gK_{(\hat\Pi)}\lambda)=gK_{(\hat\Pi)}{\rm Conv}(W_{(\hat\Pi)}\lambda)
$$
for $\hat\Pi\subset\Pi$ and $g\in K$.
\end{theorem}

With this preparation, we now turn our attention to the partial convex hulls and begin with the following elementary observation. Recall that $\ell=\dim \mathfrak t$ denotes the rank of $\mathfrak k$.

\begin{lemma}\label{Lemma ConvX is Cellplus1}
For any $\lambda\in\mathfrak t_+$ we have ${\bf r}(\lambda)\leq \ell+1$, i.e., $C_{\ell+1}(K\lambda) = {\rm Conv}(K\lambda)$.
\end{lemma}

\begin{proof}
Since ${\rm Conv}(K\lambda)=K{\rm Conv}(W\lambda)$, it suffices to show that $C_{\ell+1}(W\lambda)={\rm Conv}(W\lambda)$.

\begin{lemma}\label{Lemma Polytopes}
Let $S=\{x_1,...,x_n\}\in E$ be a finite set of points in a real vector space $E$, such that the affine hull of $S$ is the entire $E$, and let $m=\dim E$. Then $C_{m+1}(S)={\rm Conv}(S)$.
\end{lemma}

\begin{proof}
We may assume, without loss of generality, that $S$ is the set of extreme points of its convex hull $P={\rm Conv}(S)$. We have to show that every point $x\in P$ belongs to some simplex of the form ${\rm Conv}\{x_{j_1},...,x_{j_{m+1}}\}$. We shall proceed by induction on $n$. For $n=1$ the statement is trivially true. Assume that it holds for sets of cardinality $n-1$ or less. Let $F$ be any facet of $P$. Then the induction hypothesis implies $F=C_{m}(F\cap S)$. Hence, if $x_j\notin F$, then ${\rm Conv}(F\cup\{x_j\})= C_{m+1}(\{x_j\}\cup(F\cap S))$. Now let us fix one of the points, say $x_1$. Observe that $P$ is equal to the union of segments connecting $x_1$ to the boundary of $P$. Furthermore, it suffices to consider on the the facet which do not contain the fixed point. Let $F_1,...,F_k$ be the facets of $P$ which do not contain $x_1$. We have
\begin{align*}
P& ={\rm Conv}(\{x_1\}\cup F_1\cup ... \cup F_k) = \bigcup\limits_{j=1}^k {\rm Conv}(\{x_1\}\cup F_j)\\
 & = \bigcup\limits_{j=1}^k C_{m+1}(\{x_1\}\cup (F_j\cap S))\subset C_{m+1}(S) \;. 
\end{align*}
This completes the proof of Lemma \ref{Lemma Polytopes}.
\end{proof}

We apply this result to obtain
$$
{\rm Conv}(K\lambda) = K{\rm Conv}(W\lambda) = K C_{\ell+1}(W\lambda) \subset C_{\ell+1}(K\lambda) \;,
$$
and this completes the proof of Lemma \ref{Lemma ConvX is Cellplus1}.
\end{proof}

It is easy to see that $r(\lambda)=1$ if and only if ${\bf r}_0(\lambda)=1$, if and only if $\lambda=0$.

\begin{rem}\label{Rem r0 is 2}
Let $w_0\in W$ denote the longest Weyl group element, characterised by $w_0(t_+)=-\mathfrak t_+$. For $\lambda\in\mathfrak t_+$, the dominant weight $\lambda^*=-w_0\lambda$ is called the dual weight to $\lambda$, since for $\lambda\in\Lambda^+$ we have $V_{\lambda^*}\cong V_\lambda^*$. In general, $K(\lambda^*)=K(-\lambda)$ holds and implies
$$\lambda=\lambda^* \tst {\bf r}_0(\lambda)=2 \;.$$
Recall that for the simple groups of type ${\rm B}_\ell,{\rm C}_\ell,{\rm D}_{2m},{\rm E}_7,{\rm E}_8,{\rm F}_4,{\rm G}_2$ one has $w_0=-1$. Hence, in these cases, ${\bf r}_0(\lambda)=2$ for all $\lambda\in\mathfrak t_+$. Thus the behaviour of ${\bf r}_0$ is nontrivial only if $K$ has simple factors of type ${\rm A}_m, {\rm D}_{2m+1}, {\rm E}_6$.
\end{rem}

In the examples below we compute the value of ${\bf r}_0$ for all fundamental weights of classical simple groups, which are not self dual, i.e., the spin-representations of $Spin_{4\ell}$ with odd $\ell$ and the fundamental representations of $SU_{\ell+1}$.

\begin{example}\label{Exa r0 fundweights SUn}
Let $K=SU_n$ and let $\varpi_1,...,\varpi_{n-1}$ be the fundamental weights of $K$ in their standard order ($\ell=n-1$). We have
$$
{\bf r}_0(\varpi_1)=n \quad,\quad {\bf r}_0(\varpi_2)=\begin{cases}\frac{n}{2} ,\; for\; n\; even\\ \frac{n+3}{2} ,\; for\; n\; odd \end{cases} \;.
$$
If $n=jk$ with $j\leq n/2$, we have ${\bf r}_0(\varpi_j)=k$.

If $n=jq+p$ with $p\in\{0,...,j-1\}$, then one can show that ${\bf r}_0(\varpi_{j,SU_n}) = q+{\bf r}_0(\varpi_{p,SU_j})$, which by induction yields
$$
{\bf r}_0(\varpi_j) = \sum\limits_{k=1}^m q_k \;,
$$
where $m$ is the number of steps in the Euclidean algorithm for $n,j$ (at the last step we get rest 0) and $q_k$ is the integral part of the $k$-th quotient. For instance, for $j=3$, we obtain
$$
{\bf r}_0(\varpi_3) = \left\lceil \frac{n}{3} \right\rceil = \begin{cases} q \quad,\; if\; n=3q  \\ q+3 \quad,\; if\; n=3q+1\; or \; 3q+2  \end{cases} \;.
$$
\end{example}

\begin{example}\label{Exa r0 fundweights Spin2ell}
Let $K=Spin_{2\ell}$ with $\ell$ odd, $\ell\geq3$. Then
$$
{\bf r}_0(\varpi_{\ell-1})={\bf r}_0(\varpi_\ell)=4 \;.
$$
To see this, observe that the weights of the spin-representation are of the form $(\pm\frac12,\pm\frac12,...,\pm\frac12)$ with an even number of minus signs. For $\ell=3$ a minimal set $M\subset W\lambda$ whose convex hull contains $0$ is
$$
M=\{(\frac12,\frac12,\frac12),(-\frac12,-\frac12,\frac12),(-\frac12,\frac12,-\frac12),(\frac12,-\frac12,-\frac12)\} \;.
$$
For larger $\ell$ we just extend the above weights in a trivial way, two times with positive signs and two times with negative, in order to get zero as a sum. 
\end{example}



For any subset $\hat\Pi\subset\Pi$ we denote by $k'_{(\hat\Pi)}$ the semisimple part (i.e., derived subalgebra) of $\mathfrak k_{(\hat\Pi)}$. Then $\mathfrak t_{(\hat\Pi)}=\mathfrak t\cap k'_{(\hat\Pi)}={\rm span}\{\alpha^\vee:\alpha\in\hat\Pi\}$ is a maximal abelian subalgebra of $k'_{(\hat\Pi)}$. For $\lambda\in\mathfrak t_+$ we denote by $\lambda_{\vert\hat\Pi}$ is the projection of $\lambda$ to $\mathfrak t_{(\hat\Pi)}$, which defines a dominant weight of $\mathfrak k'_{(\hat\Pi)}$.

\begin{theorem}\label{Theo rlambda is max r0lambdahatPi}
For $\lambda\in\mathfrak t_+$, the following holds:
$$
{\bf r}(\lambda)=\max\{{\bf r}_0(\lambda_{\vert\hat\Pi}):\hat\Pi\subset\Pi\}\;.
$$
\end{theorem}

\begin{proof}
Recall that the center of $\mathfrak k_{(\hat\Pi)}$ is given by $\mathfrak z_{(\hat\Pi)}=\cap_{\alpha\in\hat\Pi}{\rm ker}\alpha$. Since $\lambda\in\mathfrak t\subset \mathfrak k_{(\hat\Pi)}\subset\mathfrak k$ the orbit $K_{(\hat\Pi)}\lambda$ can be seen as a coadjoint $K_{(\hat\Pi)}$-orbit. In particular, the semisimple part $K'_{(\hat\Pi)}$ acts transitively on it, and the intersection of ${\rm Conv}(K_{(\hat\Pi)}\lambda)$ with $\mathfrak z_{(\hat\Pi)}$ is a single point which we denote by $\nu_{\lambda,\hat\Pi}$, so that
$$
\{\nu_{\lambda,\hat\Pi}\}=\mathfrak z_{(\hat\Pi)}\cap{\rm Conv}(W_{(\hat\Pi)}\lambda)=\mathfrak z_{(\hat\Pi)}\cap{\rm Conv}(K_{(\hat\Pi)}\lambda) \;.
$$
Furthermore, $\nu_{\lambda,\hat\Pi}\in\mathfrak t_+$, i.e. it is dominant, and the set $E(\lambda):=\{\nu_{\lambda,\hat\Pi}:\hat\Pi\subset\Pi\}$ is exactly the extreme points of the polytope $\mathfrak t_+\cap{\rm Conv}(W\lambda)$.

For any subset $\hat\Pi\subset\Pi$, we have
$$
{\bf r}_0(\lambda_{\vert\hat\Pi})=\min\{r\in\NN:\nu_{\lambda,\hat\Pi}\in C_r(K_{(\hat\Pi)}\lambda)\}=\min\{r\in\NN:\nu_{\lambda,\hat\Pi}\in C_r(K\lambda)\},
$$
where the second equality holds, since, due to Theorem \ref{Theo BGH ConvKlambda}, ${\rm Conv}(K_{(\hat\Pi)}\lambda)$ is an exposed face of ${\rm Conv}(K\lambda)$ containing $\nu_{\lambda,\hat\Pi}$. Hence
$$
\max\{{\bf r}_0(\lambda_{\vert\hat\Pi}):\hat\Pi\subset\Pi\} = \min\{r\in\NN: E(\lambda)\subset C_r(K\lambda)\} \leq {\bf r}(\lambda) \;.
$$

On the other hand, recall the momentum map
$$
\mu^{\iota_r(\lambda)}: (K\lambda)^{\times r} \to \mathfrak k, \; (x_1,...,x_r)\mapsto\frac1r(x_1+...+x_r) \;,
$$
whose image $\mu^{\iota_r(\lambda)}((K\lambda)^{\times r})=\frac1r(K\lambda+...+K\lambda)$ is contained in $C_r(K\lambda)$. Since the intersection of the momentum image with the Weyl chamber $\mathfrak t_+$ is convex, to prove the theorem it is sufficient to show that $E(\lambda)\subset \mu^{\iota_r(\lambda)}((K\lambda)^{\times r})$ for $r=\max\{{\bf r}_0(\lambda_{\vert\hat\Pi}):\hat\Pi\subset\Pi\}$. From Theorem \ref{Theo Cr cones} we know that $r\geq{\bf r}_0(\lambda)$ is equivalent to $0\in \frac1r(K\lambda+...+K_\lambda)$. Applying this to each $K_{(\hat\Pi)}\lambda_{\vert\hat\Pi}$ we get $r\geq{\bf r}_0(\lambda_{\vert\hat\Pi})$ if and only if $\nu_{\lambda,\hat\Pi}\in \frac1r(K_{(\hat\Pi)}\lambda_{\vert\hat\Pi}+...+K_{(\hat\Pi)}\lambda_{\vert\hat\Pi})$, which yields the desired result.
\end{proof}

Now we are ready to prove Theorem \ref{Theo Theo Er cones} stated in the Introduction.

\begin{proof}[of Theorem \ref{Theo Theo Er cones}]
We shall first show the equivalence of the three conditions. Note that condition (iii) of the theorem can be rewritten with the notation introduced in this section as: for every subset $\hat\Pi\subset\Pi$, $\lambda_{\vert\hat\Pi}\in\mathcal{CLR}_{r}(K'_{(\hat\Pi)})$. From Theorem \ref{Theo Cr cones} we know that $\lambda_{\vert\hat\Pi}\in\mathcal{CLR}_{r}(K'_{(\hat\Pi)})$ is equivalent to $r\geq{\bf r}_0(\lambda_{\vert\hat\Pi})$. Thus condition (iii) is equivalent to $r\geq {\bf r}_0(\lambda_{\vert\hat\Pi})$ for all $\hat\Pi\subset\Pi$, which is in turn equivalent to $r\geq{\bf r}(\lambda)$ due to Theorem \ref{Theo rlambda is max r0lambdahatPi}. Therefore conditions (i) and (iii) are equivalent.

The fact that (ii) implies (i) follows immediately from the observation that the rescaled $r$-fold sum $\frac1r(K\lambda+...+K\lambda)$ is contained in $C_r(K\lambda)$ and contains $K\lambda$, hence if the sum is convex then so is the partial convex hull and we must have $r\geq {\bf r}(\lambda)$.

We suppose now that (i) (and hence (iii)) holds and shall verify (ii), or more precisely, that the normalized $r$-fold sum $A_r=\frac1r(K\lambda+...+K\lambda)$ is convex. By Theorem \ref{Theo Cr cones}, $0\in C_r(K\lambda)$ is equivalent to $0\in A_r$. Similarly, from the proof of Theorem \ref{Theo rlambda is max r0lambdahatPi} and with the notation introduced therein, we deduce that $\nu_{\lambda,\hat\Pi}\in C_r(K\lambda)$ is equivalent to $\nu_{\lambda,\hat\Pi}\in C_r(K_{(\hat\Pi)}\lambda)$, which is in turn equivalent to $\nu_{\lambda,\hat\Pi}\in \frac1r(K_{(\hat\Pi)}\lambda+...+K_{(\hat\Pi)}\lambda)\subset A_r$. Thus (i) implies $\nu_{\lambda,\hat\Pi}\in A_r$ for all $\hat\Pi\subset\Pi$. The convex hull of these points is equal to $\mathfrak t_+\cap{\rm Conv}(K\lambda)$, and $A_r$ must contain this convex hull, because $A_r$ is equal to the momentum image $\mu^{\iota_r(\lambda)}(K\lambda^{\times r})$ and hence its intersection with the Weyl chamber is a convex polytope. Thus $A_r={\rm Conv}(K\lambda)$ and property (ii) holds.

It remains to show that $\mathfrak C_r$ is a rational convex polyhedral cone. It is not difficult to see that, for any $\hat\Pi\subset\Pi$, the set
$$
\mathfrak B_{r,\hat\Pi} = \{(\lambda_1,...,\lambda_r)\in(\mathfrak t_+)^{\times r}:((\lambda_1)_{\vert\hat\Pi},...,(\lambda_r)_{\vert\hat\Pi})\in\mathcal{CLR}_r(K'_{(\hat\Pi)})\}
$$
is a rational polyhedral convex cone described analogously to the Littlewood-Richardson cone. Now the equivalence between (i) and (iii) yields
$$
\mathfrak C_r=\iota_r(\mathfrak t_+)\cap\left(\bigcap\limits_{\hat\Pi\subset\Pi} \mathfrak B_{r,\hat\Pi} \right)
$$
and we can deduce the claimed properties of $\mathfrak C_r$.
\end{proof}

As a special case, we have the following.

\begin{coro}
For $\lambda\in\mathfrak t_+\setminus\{0\}$ the following are equivalent:
\begin{enumerate}
\item[{\rm (i)}] ${\bf r}(\lambda)=2$;
\item[{\rm (ii)}] $K\lambda+K\lambda$ is convex;
\item[{\rm (iii)}] for any subset $\hat\Pi\subset\Pi$, ${\bf r}_0(\lambda_{\vert\hat\Pi})=2$, i.e., $K'_{(\hat\Pi)}\lambda_{\vert\hat\Pi}\ni-\lambda_{\vert\hat\Pi}$;
\item[{\rm (iv)}] For any connected component $\Pi_1$ of (the Dynkin diagram of) $\Pi$ and any $\alpha,\beta\in\Pi_1$ such that $||\alpha||=||\beta||$, the equality $(\lambda|\alpha)=(\lambda|\beta)$ holds.
\end{enumerate}
\end{coro}

\begin{example}
Let $K$ be a simply connected classical group and $\varpi_1,...,\varpi_\ell$ be the fundamental weights of $K$ with the standard ordering of \cite{Bourbaki-Lie-2}. The values of ${\bf r}$ at the fundamental weights can be computed using Theorem \ref{Theo rlambda is max r0lambdahatPi} and the values of ${\bf r}_0$ at fundamental weights computed in Remark \ref{Rem r0 is 2} and Examples \ref{Exa r0 fundweights SUn}, \ref{Exa r0 fundweights Spin2ell}. Here we use the fact that the simple factors of centralizer subgroups in classical groups are again classical groups. We obtain the following.

For $K=SU_{\ell+1}$ and $1\leq j\leq \frac{\ell+1}{2}$ we have ${\bf r}(\varpi_j)={\bf r}(\varpi_{\ell+1-j})=(\ell+1)-(j-1)$.

For $K=Spin_{2\ell+1}$, or $Sp_{2\ell}$, we have ${\bf r}(\varpi_j)={\bf r}(\varpi_{\ell-j})=\ell-(j-1)$, for $1\leq j\leq \frac{\ell}{2}$, and ${\bf r}(\varpi_\ell)=2$.

For $K=Spin_{2\ell}$ we have ${\bf r}(\varpi_j)={\bf r}(\varpi_{\ell-j})=\ell-(j-1)$, for $1\leq j\leq \frac{\ell}{2}$, and ${\bf r}(\varpi_\ell)=\ell$.

We observe that ${\bf r}(\varpi_j)$ is equal to $1+\hat\ell_j$, where $\hat\ell_j$ is the length of the longest chain of type A in $\Pi$ emanating from $\alpha_j$. The number $\hat\ell_j$ is also equal to the dimension of the largest projective space $\PP$ equivariantly embedded in $K\varpi_j$ and mapped to a linear subspace in $\PP(V_{\omega_j})$ when $K\varpi_j$ is mapped (isomorphically) to the orbit $K[v_{\varpi_j}]$ of a highest weight line (see \cite{Lands-Mani-2003-ProjGeo}).
\end{example}

\begin{rem}
Lists of inequalities describing the cones $\mathfrak A_r$ and $\mathfrak C_r$ (see Theorems \ref{Theo Theo Cr cones} and \ref{Theo Theo Er cones}) can be deduced form the known lists of inequalities describing $\mathcal{CLR}_r$. A minimal list of inequalities for $\mathcal{CLR}_r$ was obtained by Belkale and Kumar in \cite{Belkale-Kumar}. The resulting inequalities for $\mathfrak A_r$ have the form
$$
(\lambda|w_1\varpi_j+...+w_r\varpi_j) \geq 0 \;,
$$
where $\varpi_j$ is a fundamental weight of $K$ and $w_1,...,w_r\in W$ satisfy the so-called Belkale-Kumar condition (cf. \cite{Belkale-Kumar},\cite{Brion-Survey}).
\end{rem}

\vspace{0.4cm}

\author{\noindent Valdemar V. Tsanov \\
Institute of Mathematics and Informatics, Bulgarian Academy of Sciences\\
ul. ``Acad. Georgi Bonchev'' 8, 1113 Sofia, Bulgaria.\\
Email: valdemar.tsanov@math.bas.bg}

\end{document}